\documentclass[leqno,11pt]{amsart}
\usepackage{amsmath}
\usepackage{amsfonts}
\usepackage{amssymb}
\usepackage{amsthm}
\usepackage{enumerate}
\usepackage{mathrsfs}
\usepackage{dsfont}
\usepackage{txfonts}
\usepackage{mathabx}
\usepackage{comment}

\usepackage{enumerate}

\usepackage{graphicx}
\usepackage[all]{xy}
\usepackage{pgf,tikz,pgfplots}
\usepackage{subfigure}
\usepackage{tikz-3dplot}
\usetikzlibrary{patterns,arrows}
\tikzset{
    partial ellipse/.style args={#1:#2:#3}{
        insert path={+ (#1:#3) arc (#1:#2:#3)}
    }
}

\usepackage[pdftex]{hyperref}

\usepackage{color}
\usepackage[shortlabels]{enumitem}

\newtheorem{theorem}{Theorem}[section]
\newtheorem{principle}[theorem]{Principle}
\newtheorem{lemma}[theorem]{Lemma}
\newtheorem{corollary}[theorem]{Corollary}

\newtheorem{proposition}[theorem]{Proposition}

\newtheorem{claim}[theorem]{Claim}

\newtheorem*{claim*}{Claim}

\theoremstyle{definition}
\newtheorem{definition}[theorem]{Definition}

\theoremstyle{remark}
\newtheorem{remark}[theorem]{Remark}

\newcommand{\dist}{\mathrm{dist}}

\newcommand{\ee}{\mathbf{e}}

\newcommand{\eps}{\varepsilon}
\newcommand{\Hh}{\mathcal{H}}
\newcommand{\reg}{\operatorname{reg}}
\newcommand{\sing}{\operatorname{sing}}
\newcommand{\spt}{\operatorname{spt}}
  \newcommand{\Div}{\operatorname{Div}}
   \newcommand{\Cc}{\mathcal C}
    \newcommand{\Tan}{\operatorname{Tan}}

\newcommand{\RR}{\mathbb{R}}

\begin{document}

\title{Moving plane method for varifolds and applications}

\author{Robert Haslhofer, Or Hershkovits, Brian White}
\dedicatory{Dedicated to the memory of Professor Louis Nirenberg}

\begin{abstract} In this paper, we introduce a version of the moving plane method that applies to potentially quite singular hypersurfaces, generalizing the classical moving plane method for smooth hypersurfaces. Loosely speaking, our version for varifolds shows that smoothness and symmetry at infinity (respectively at the boundary) can be promoted to smoothness and symmetry in the interior. The key feature, in contrast with the classical 
formulation of the moving plane principle, is that smoothness is a conclusion rather than an assumption.

We implement our moving plane method in the setting of compactly supported varifolds with smooth boundary and in the setting of  varifolds without boundary.  A key ingredient is a Hopf lemma for stationary varifolds and  varifolds
of constant mean curvature. Our Hopf lemma provides a  new tool to establish smoothness of varifolds, and works in arbitrary dimensions and without any stability assumptions. 
As applications of our new moving plane method, we prove 
varifold uniqueness results for the catenoid, spherical caps, and Delaunay surfaces that are inspired by
 classical uniqueness results by Schoen, Alexandrov, Meeks and Korevaar-Kusner-Solomon. We also prove a varifold version of Alexandrov's Theorem
  for compactly supported varifolds of constant mean curvature in hyperbolic space.  
\end{abstract}

\maketitle

\section{Introduction}

The moving plane method is a fundamental tool to establish symmetry in geometry and partial differential equations. 
This method was pioneered by Alexandrov \cite{Alexandrov}, Serrin \cite{Serrin}, Gidas-Ni-Nirenberg \cite{GidasNiNirenberg}, Berestycki-Nirenberg \cite{BerestyckiNirenberg}, and Schoen \cite{Schoen}. Since then, the method has been applied frequently in the literature.  For surveys and comprehensive references see e.g. Brezis \cite{Brezis} and Ciraolo-Roncoroni \cite{CiraoloRoncoroni}. Loosely speaking, the method gives rise to the following general principle:
 
\begin{principle}[Classical moving plane principle] Symmetry at infinity (or at the boundary) can be promoted to symmetry in the interior.
\end{principle}

In our recent joint work with Choi, we discovered a new variant of the moving plane principle, which can be loosely speaking stated as follows:

\begin{principle}[{New moving plane principle \cite{CHHW}}]\label{principle_new}
Smoothness and symmetry at infinity (or at the boundary) can be promoted to smoothness and symmetry in the interior.
\end{principle}

\begin{figure}
\begin{tikzpicture}[x=1cm,y=1cm] \clip(-8,12) rectangle (8,-1);
\shade[left color=blue!50!,right color=blue!10!] (4*1,-4*1) rectangle (4*1.2,4*4);
\draw [samples=100,rotate around={0:(0,0)},xshift=0cm,yshift=0cm,domain=4*0.8245253419765505:8)] plot (\x,{(\x)^2/4});
\draw [dashed] (4*0.305,4*0.25) [partial ellipse=0:180:4*0.18cm and 4*0.03cm];
\draw (4*0.305,4*0.25) [partial ellipse=180:360:4*0.18cm and 4*0.03cm];
\draw [dashed,rotate around={120:(-4*0.64,4*0.35)}] (-4*0.64,4*0.35) [partial ellipse=180:360:4*0.06cm and 4*0.01cm];
\draw [rotate around={120:(-4*0.64,4*0.35)}] (-4*0.64,4*0.35) [partial ellipse=0:180:4*0.06cm and 4*0.01cm];
\draw [dotted,thick,rotate around={240:(4*0.82,4*0.64)}] (4*0.82,4*0.64) ellipse (4*0.024cm and 4*0.004cm);
\draw [dashed] (-4*0.22,4*0.22) [partial ellipse=0:180:4*0.21cm and 4*0.035cm];
\draw (-4*0.22,4*0.22) [partial ellipse=180:360:4*0.21cm and 4*0.035cm];
\draw [dashed] (4*0.0002332491688299718,4*1.5997785993677849) [partial ellipse=0:180:4*1.2571684499174602cm and 4*0.21455991555378162cm];
\draw (4*0.0002332491688299718,4*1.5997785993677849) [partial ellipse=180:360:4*1.2571684499174602cm and 4*0.21455991555378162cm];
\draw [dashed] (4*0.00032078487570581966,4*2.1966671075002657) [partial ellipse=0:180:4*1.4743165895739172cm and 4*0.29007806481583676cm];
\draw (4*0.00032078487570581966,4*2.1966671075002657) [partial ellipse=180:360:4*1.4743165895739172cm and 4*0.29007806481583676cm];
\draw   (4*0.8245253419765505,4*0.6798420395615475)-- (4*0.95,4*0.6);
\draw   (4*0.7819932155552015,4*0.6115133891743638)-- (4*0.95,4*0.6);
\draw   (4*0.648407203918607,4*0.4204319020935459)-- (4*0.6423844814564859,4*0.37711740705108604);
\draw [samples=100,rotate around={0:(0,0)},xshift=0cm,yshift=0cm,domain=-8:-4*0.6408712864207166)] plot (\x,{(\x)^2/4});
\draw   (-4*0.3646156662354182,4*0.08302374114334066)-- (-4*0.31653544754347496,4*0.10019468955154799);
\draw   (-4*0.6408712864207166,4*0.4107160057585441)-- (-4*1,4*0.3);
\draw   (-4*1,4*0.3)-- (-4*0.53991811338522,4*0.29151156916145526);
\draw [samples=100,rotate around={0:(0,0)},xshift=0cm,yshift=0cm,domain=-4*0.31653544754347496:-4*0.20622200492137188)] plot (\x,{(\x)^2/4});
\draw [shift={(-4*0.7027681841100468,-4*0.022034854577463753)}]  plot[domain=0.3:1.0917643671335866,variable=\t]({4*1*0.3533150228543174*cos(\t r)+4*0*0.3533150228543174*sin(\t r)},{4*0*0.3533150228543174*cos(\t r)+4*1*0.3533150228543174*sin(\t r)});
\draw   (-4*0.20622200492137188,4*0.04252751531379033)-- (4*0.1,4*0.3);
\draw   (4*0.1,4*0.3)-- (4*0.2130190048395076,4*0.04537709642281416);
\draw [samples=100,rotate around={0:(0,0)},xshift=0cm,yshift=0cm,domain=4*0.2130190048395076:4*0.5199513674882179)] plot (\x,{(\x)^2/4});
\draw   (4*0.5797243057185343,4*0.3114258682101738)-- (4*0.5199513674882179,4*0.2703494245528678);
\draw [shift={(4*0.6745638963495582,4*0.2836928070184077)}]  plot[domain=1.9:2.8571037515665014,variable=\t]({4*1*0.09881128798941159*cos(\t r)+4*0*0.09881128798941159*sin(\t r)},{4*0*0.09881128798941159*cos(\t r)+4*1*0.09881128798941159*sin(\t r)});
\draw [shift={(4*0.5386990551836357,4*0.6393654747679093)}]  plot[domain=-1.106267150134368:-0.1,variable=\t]({4*1*0.24488321123101786*cos(\t r)+4*0*0.24488321123101786*sin(\t r)},{4*0*0.24488321123101786*cos(\t r)+4*1*0.24488321123101786*sin(\t r)});
\draw   [color=blue, dashed] (4*1.2,-4*1) -- (4*1.2,4*4);
\draw [<-, color=blue] (4*0.5,4*1.1) -- (4*0.9,4*1.1);
\begin{scriptsize}
\draw[color=blue] (4*0.7,4*1.2) node[scale=1.3] {$moving\; plane$};
\end{scriptsize}
\end{tikzpicture}
\caption{The cap region could a priori be quite singular. Whenever the moving plane reaches a point (away from the axis of symmetry) then this point must be smooth.}\label{amazing_figure}
\end{figure}
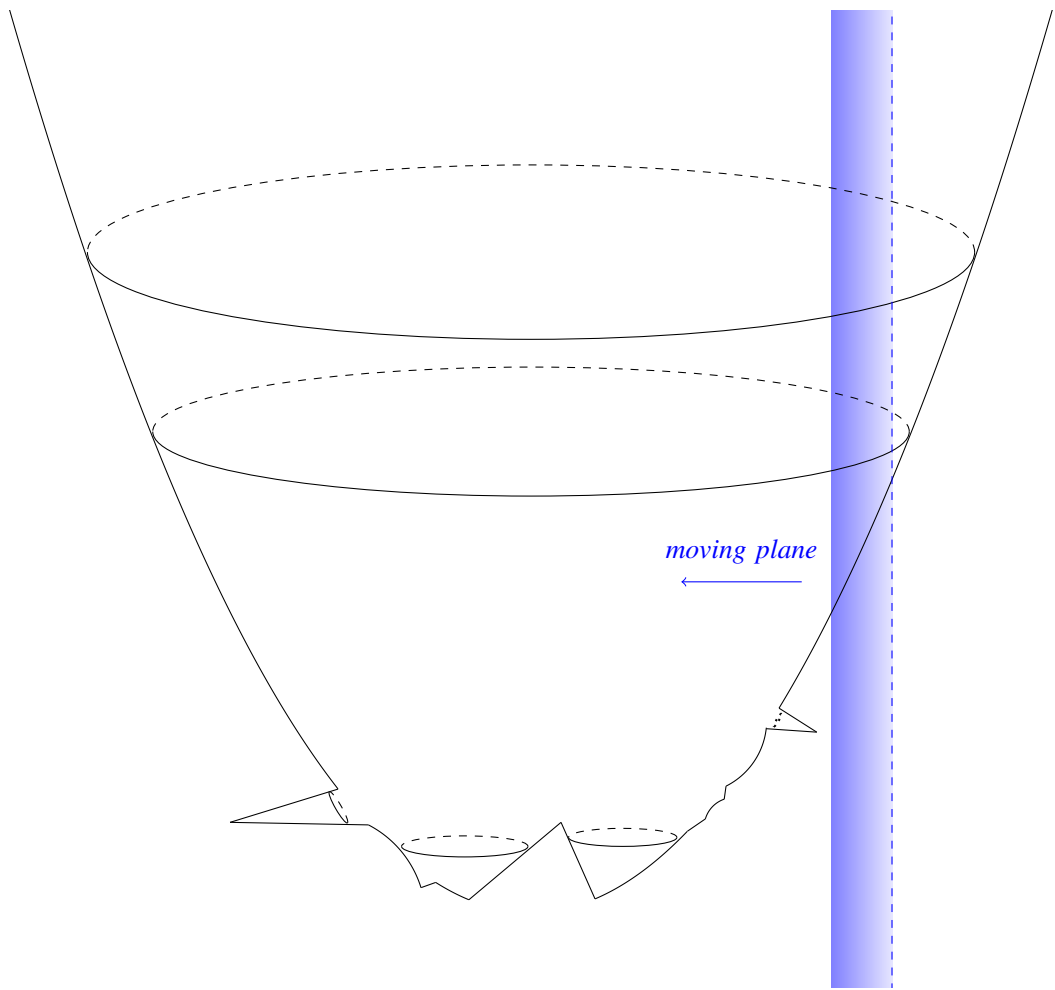

The gist of this principle is illustrated in Figure \ref{amazing_figure}. The key feature, in contrast with the classical moving plane principle, is that smoothness is a conclusion, not an assumption. This was of crucial importance in our recent proof of the  canonical neighborhood conjecture for mean curvature flow through neck-singularities in arbitrary dimensions \cite{CHHW}. There, the blowup limits that we encountered near a neck-singularity were potentially quite singular objects (integral Brakke flows satisfying a few technical conditions), and in order to classify them we had to establish smoothness and symmetry in tandem. Unfortunately, the parabolic (rather than elliptic) nature of \cite{CHHW}, together with the many other ingredients of that paper, may obscure the simplicity and generality of Principle~\ref{principle_new}.\\

In the present paper, which we hope is accessible to a larger class of readers, we develop the method in the elliptic setting of potentially singular hypersurfaces. Specifically, we implement the method to prove uniqueness results for stationary varifolds and CMC-varifolds,
inspired by classical uniqueness results for smooth hypersurfaces by Schoen \cite{Schoen}, Alexandrov \cite{Alexandrov}, Meeks \cite{Meeks} and Korevaar-Kusner-Solomon \cite{KKS}. These uniqueness results will be discussed in Section \ref{intro_stationary} and Section \ref{intro_CMC}, respectively. A description of our new method itself, including in particular a Hopf lemma for varifolds -- which seems to be of independent interest -- will be given in Section \ref{intro_Hopf}.

\bigskip

We recently learned that Bernstein-Maggi \cite{BernsteinMaggi} have developed a very interesting variant of the moving plane method that allows them to prove symmetry of certain singular Plateau surfaces, and seems rather different from our method which establishes smoothness.

\bigskip

\subsection{Uniqueness results for stationary varifolds}\label{intro_stationary}

In this section, we describe our uniqueness results for potentially singular minimal hypersurfaces in $\mathbb{R}^{n+1}$. These potentially singular objects are described most conveniently as stationary integral $n$-varifolds, as introduced by Almgren and Allard~\cite{Almgren,Allard}. An $n$-varifold is a measure-theoretic generalization of an $n$-dimensional surface. Integral means that at almost every point one can find a tangent plane of integer multiplicity, and stationary means that the varifold is a critical point of the $n$-dimensional area functional (see Section \ref{sec_prelim} for precise definitions and notation).\\

The moving plane method for noncompact hypersurfaces was pioneered by Schoen~\cite{Schoen}, who proved uniqueness of the catenoid among smooth minimal hypersurfaces that are asymptotic to two planes. Our first main theorem extends Schoen's uniqueness result to the setting of varifolds:

\begin{theorem}[Uniqueness of the catenoid]\label{main_complete_theorem}
Suppose $M$ is stationary integral $n$-varifold in $\mathbb{R}^{n+1}$ such that
\begin{enumerate}
\item some tangent cone at infinity is a multiplicity-$2$ plane,
\item $M$ has at least two ends,\footnote{By the convex hull property, $M$ has at least two ends if and only if there is a ball $B$ such that $M\setminus B$ is not connected.} 
\item $M$ has no triple junctions.
\end{enumerate}
Then $M$ is a smooth hypersurface of revolution, and thus either a pair of parallel planes or an $n$-dimensional catenoid.
\end{theorem}

A triple junction is a point at which one of the tangent cones 
consists of three multiplicity-one halfplanes meeting at equal angles along their common edge.  We remark that, according to \cite{simon-cylindrical}, the tangent cone
to a stationary varifold at a triple junction point is unique, and in a small neighborhood of the point, the varifold is diffeomorphic to its tangent cone.  (The results of 
  \cite{simon-cylindrical} are not needed in this paper.)

The main novelty is that smoothness is a conclusion of our theorem rather than an assumption. In the study of minimal hypersurfaces, one often passes to weak limits, which could be potentially quite singular. Theorem \ref{main_complete_theorem} does not assume smoothness a priori, and thus can be used to analyze such limits.

In the same paper \cite{Schoen}, Schoen also proved symmetry for smooth compact minimal hypersurfaces with rotationally symmetric boundary. We extend this result as well:

\begin{theorem}[smoothness and symmetry for stationary varifolds with smooth boundary]\label{main_boundary_theorem}
Suppose that $\Gamma=\Gamma_1\cup\Gamma_2$ is the union of two smooth, closed, strictly convex $(n-1)$-dimensional manifolds lying in two parallel hyperplanes $Q_1$ and $Q_2$
in $\RR^{n+1}$.
Suppose $M$ is an integral $n$-varifold in $\RR^{n+1}$ 
with compact, connected support
 such that
\begin{enumerate}[\upshape (i)]
\item\label{annulus-stationary-hypothesis} $M$ is stationary in $\RR^{n+1}\setminus \Gamma$, 
\item\label{annulus-density-hypothesis} The density of $M$ at each point of $\Gamma$ is $1/2$, 
\item\label{annulus-triple-hypothesis} $M$ has no triple junctions.
\end{enumerate}
If $P$ is a hyperplane of symmetry of the boundary $\Gamma$, and if $P$ is perpendicular to hyperplanes $Q_i$, then $P$ is a hyperplane of symmetry of $M$.
 Moreover, the portion of $M$ on each side of $P$ is a smooth graph 
 over a region in $P$.

In particular, if $\Gamma$ is rotationally symmetric, 
 then $M$ is a piece of an $n$-dimensional catenoid.
\end{theorem} 

An interesting feature of the theorem is that it gives smoothness (away from the plane of symmetry) even if the boundary is symmetric only with respect to a single plane $P$, and not necessarily rotationally symmetric. Combining Theorem~\ref{main_boundary_theorem}  with some results from \cite{white-boundary-mcf} we also obtain:

\begin{corollary}\label{GMT_cor}
If the hypotheses~\ref{annulus-stationary-hypothesis},
 \ref{annulus-density-hypothesis} and~\ref{annulus-triple-hypothesis} are replaced by the
 hypotheses
\begin{enumerate}[\upshape (i$'$)]
\item\label{div-inequality} $\int \Div_M X\,d\mu_M \le \int_\Gamma|X|\,d\Hh^{m-1}$ 
for every $C^1$ vectorfield $X$ on $\RR^{n+1}$, and
\item\label{mod-2-boundary} The boundary of the mod $2$ flat chain associated to $M$ is $[\Gamma]$,
\end{enumerate}
then the conclusion of Theorem~\ref{main_boundary_theorem} holds.
\end{corollary}

For readers not familiar with mod $2$ flat chains, Hypothesis~\ref{mod-2-boundary} is equivalent to the statement:
If $C$ is a smooth simple closed curve in $\RR^{n+1}$, then for almost every $v\in \RR^{n+1}$, 
the number of intersections of $(C+v)\cap M$ (counting multiplicity) is equal to the mod $2$ linking number
of $C+\nu$ and $\Gamma$.

(The proof of Corollary~\ref{GMT_cor} is given at the end of
 Section~\ref{sec_boundary}.)

\bigskip

\subsection{Uniqueness results for varifolds of constant mean curvature}\label{intro_CMC}
In this section, we consider possibly singular hypersurfaces with constant mean curvature.  Variationally, such varifolds arise as critical points of the $n$-dimensional area functional subject to the constraint that the enclosed volume is kept fixed. Our results in this section require the following regularity assumption.\footnote{In the stationary case, we will see that tameness can be deduced, using monotonicity, from the assumptions of Theorem \ref{main_complete_theorem} and Theorem \ref{main_boundary_theorem}, respectively. However,  in the CMC case, tameness does not follow readily from  global assumptions.}

\begin{definition}[tameness]\label{def_tame}
An $n$-dimensional integral varifold $M$ in an $(n+1)$-dimensional Riemannian
manifold $N$ is called
\emph{tame} (in $N$) provided there is a smooth
$(n-1)$-dimensional manifold $\Gamma$ (called the \emph{boundary} of $M$) such that
\begin{enumerate}[\upshape (i)]
\item $M$ has constant mean curvature in $N\setminus \Gamma$,
\item At each point in $ M\setminus\Gamma$, each tangent cone is smooth with multiplicity one away from a set of $(n-1)$-dimensional
Hausdorff measure $0$, and 
\item\label{density-one-half-item} At each point of $\Gamma$, the density of $M$ is $1/2$.
\end{enumerate}
\end{definition}

By the Allard Boundary Regularity Theorem~\cite{Allard_boundary},  Condition~\ref{density-one-half-item} 
  is equivalent to the condition that, near $\Gamma$, $M$ is a smooth multiplicity-one
manifold-with-boundary, the boundary being $\Gamma$. 

The moving plane method was introduced by Alexandrov \cite{Alexandrov}, who proved that every compact connected smooth embedded CMC hypersurface must be a round sphere.  A related classical result of Meeks \cite{Meeks} states that there is no noncompact smooth embedded CMC surface in $\mathbb{R}^3$ with a single end. Inspired by these results, we have the following uniqueness theorem for spherical caps:

\begin{theorem}[Uniqueness of CMC spherical caps]\label{CMC_uniq}
Suppose that $M$ is a tame CMC varifold in $\mathbb{R}^{n+1}$ with compact, connected support. 
Suppose also that the boundary of $M$ is an $(n-1)$-sphere in an $n$-plane $P$,
and that $M$ lies on one side of $P$ and meets $P$ transversally.  Then $M$ is a portion of a round $n$-sphere. 
\end{theorem}

We remark that for the Alexandrov theorem itself, there is a beautiful recent result by Delgadino-Maggi \cite{DelgadinoMaggi}, which establishes uniqueness among compact connected sets of finite perimeter without tameness assumption. Their proof relies on a clever use of the Heintze-Karcher inequality and generalizes an earlier argument by Montiel-Ros \cite{MontielRos}.

The main advantage of our own method is that it seems to have a much wider scope. In particular, we can deal with boundaries and with situations with much less symmetries. Moreover, our method also applies to certain other ambient manifolds as illustrated by the following theorem:

\begin{theorem}\label{alex_gen_thm}
Let $M$ be a tame CMC varifold with compact, connected support in an $(n+1)$-dimensional Riemannian manifold $N$. If $N$ is Euclidean space or hyperbolic space, then M is a geodesic sphere. If $N$ is ($n$-dimensional hyperbolic space)$\times \mathbb{R}$, then $M$ is smooth and rotationally invariant about an axis $\{p\}\times \mathbb{R}$.
\end{theorem}

Regarding CMC varifolds with two boundary components, 
the following is inspired by the classical uniqueness theorem for Delaunay strips by Korevaar-Kusner-Solomon~\cite{KKS}:

\begin{theorem}[Uniqueness of Delaunay strips]\label{main_boundary_CMC_theorem}
Let $M$ be a  tame, constant mean curvature varifold in $\mathbb{R}^{n+1}$ with compact, connected support and
with boundary consisting of a pair of $(n-1)$-spheres that lie in parallel $n$-planes
and that have the same axis of rotational symmetry.
 Suppose that  $M$ is contained in the slab between the $n$-planes, and meets these planes transversally.
Then $M$ is a smooth embedded hypersurface of revolution, and hence a piece of a Delaunay hypersurface.
\end{theorem}

\bigskip

\subsection{Moving plane method for varifolds}\label{intro_Hopf}
Recall that the classical moving plane method relies on the maximum principle and the Hopf Lemma in the smooth setting. To implement our new variant of the moving plane method, we need generalizations of the maximum principle and the Hopf lemma to the varifold setting.\\

To discuss our Hopf lemma for varifolds, denote by $\mathbb{H}\subset \mathbb{R}^{n+1}$ an open halfspace whose boundary $n$-plane contains the origin. Recall first that the classical Hopf lemma says that if $u_1,u_2$ are smooth solutions of a second order elliptic partial differential equation, defined in an open ball $B(0,r)$, such that $u_1(0)=u_2(0)$ and $u_1(x)<u_2(x)$ for all $x\in B(0,r)\cap\mathbb{H}$, then $u_1$ and $u_2$ have distinct normal derivatives at $0$.
The first guess regarding how to generalize this for varifolds would be to infer that the tangent cones at $0$ must be distinct. But actually our conclusion is much better. In essence, we can use the fact that one varifold lies above the other one to \emph{conclude} that $0$ must be a smooth point. Specifically, we prove:

\begin{theorem}[Hopf lemma for varifolds]\label{mirror-theorem_intro}
Let $M_1$ and $M_2$ be integral $n$-varifolds in $B(0,r)\subset\RR^{n+1}$ that have the same constant mean curvature $h$ (possibly $0$).
If
\begin{enumerate}
\item  $0\in M_1\cap M_2$ is a tame interior point for both $M_1$ and $M_2$,
\item $\partial \mathbb{H}$ is \textbf{not} the tangent cone to either $M_1$ or $M_2$ at $0$,
\item and $\reg M_1\cap \mathbb{H}$ and $\reg M_2\cap \mathbb{H}$ are disjoint,
\end{enumerate}
then either
\begin{enumerate}[(a)]
\item $h\ne 0$ and $M_1$ and $M_2$ curve oppositely at $0$, or
\item $M_1$ and $M_2$ are smooth at $0$, with distinct tangents.
\end{enumerate}
\end{theorem}

The key feature of Theorem \ref{mirror-theorem_intro} is that smoothness is a conclusion and not an assumption. It thus provides a new tool to establish regularity of varifolds. Moreover, a quite unique advantage is that this works without any dimensional restrictions and stability conditions (in contrast to other available methods where one usually assumes stability and $n<7$). For our new moving plane method, the case of particular interest is when $M_2$ is the image of $M_1$ under reflection
in the plane $\partial \mathbb{H}$.\\

Let us now explain some technical details: For a stationary or CMC integral $n$-varifold $M$, as in the previous subsection, a point $x\in M$ is called \emph{tame} if each tangent cone of  $M$ at $x$ is smooth with multiplicity one away from a set of $(n-1)$-dimensional Hausdorff measure zero. 
This tameness condition cannot be dropped, since the statement clearly fails for triple-junctions. Next, $\reg M$ simply denotes the set of regular points. Finally, if $M_1$ and $M_2$ have the same constant mean curvature $h$, we say that $M_1$ and $M_2$ \emph{curve oppositely} at a common  tame point $x$ if $M_1$ and $ M_2$ have a mutual tangent cone at $x$, but with the opposite orientation (see Definition \ref{def_opp_curv}). This scenario can often be ruled out in applications, e.g. when $M_a$ arises as the boundary of some domain $K_a$.\\

Next, let us briefly discuss the maximum principle. Fortunately, nonsmooth versions of the maximum principle for hypersurfaces have been studied extensively in the literature (in stark contrast to the Hopf lemma). The sharpest result has been obtained by Wickramasekera \cite{Wick_max}, building on earlier work by Simon \cite{Simon_max} and Ilmanen \cite{Ilmanen_max}. For our purpose, the following simple variant by Solomon and the last author, where one of the two varifolds is assumed to be smooth, is sufficient:

\begin{theorem}[{Maximum principle for varifolds \cite{SolomonWhite}}]\label{strong_max_varifold_intro}
Let $M_1$ be a smooth minimal hypersurface defined in a ball $B(0,r)$ centered at $0\in M_1$ and with small enough radius $r>0$ such that $ M_1$ separates $B(0,r)$ into two open connected components, $U$ and $U'$.
Let $M_2$ be a stationary integral $n$-varifold in $B(0,r)$, and assume that $0\in M_2$ is a tame interior point. If
\begin{equation}\label{max_princ_cont_intro}
 M_2 \subseteq U\cup M_1,
\end{equation}
then $0$ is a smooth point for $M_2$ and there exists some $\eps>0$ such that 
\begin{equation}\label{conclusion_strong_max_intro}
 M_2 \cap B(0,\eps) = M_1 \cap B(0,\eps).
\end{equation}
The same conclusion holds in the CMC-case, unless $M_1$ and $M_2$ curve oppositely at $0$.
\end{theorem}

This is well-suited for our purpose, since smoothness of $M_2$ is not an assumption but a conclusion. Having discussed these tools, let us now describe the moving plane method for varifolds: \\

\noindent\textbf{Description of the method.} Suppose we would like to establish smoothness and symmetry of a stationary or CMC varifold $M$, with respect to a plane $P$. Denote by $P_t$ the plane parallel to $P$ at level $t$ above $P$. The idea is then to push down the plane $P_t$, starting from a sufficiently high level, to establish smoothness and symmetry in tandem. More precisely, we want to show that for all $t>0$ we simultaneously have:
\begin{enumerate}[(i)]
\item $M$ can be reflected across $P_t$, \emph{and}
\item the part of $M$ that lies above $P_t$ is smooth.
\end{enumerate}

We first check that this is true for sufficiently large $t$, either by compactness or by suitable asymptotics. Suppose towards a contradiction we get stuck at some level $t>0$. We then argue that this contradicts either the maximum principle for varifolds (Theorem \ref{strong_max_varifold_intro}) or the Hopf lemma for varifolds (Theorem \ref{mirror-theorem_intro}). We can thus push the moving plane all the way to level $0$, which establishes smoothness and symmetry.

\bigskip

This article is organized as follows. In Section \ref{sec_prelim}, we summarize some preliminaries about varifolds. In Section \ref{sec_hopf}, we prove our Hopf lemma for varifolds (Theorem \ref{mirror-theorem_intro}). In Section \ref{sec_boundary}, we prove our results for compact varifolds with boundary, specifically Theorem \ref{main_boundary_theorem}, Theorem \ref{CMC_uniq} and Theorem \ref{main_boundary_CMC_theorem}. In Section \ref{sec_complete}, we address the noncompact case and prove uniqueness of the catenoid (Theorem \ref{main_complete_theorem}). Finally,  in Section \ref{sec_man}, we explain how the method can be used in more general ambient manifolds, and prove Theorem \ref{alex_gen_thm}.

\bigskip

\noindent\textbf{Acknowledgments.} RH has been partially supported by an NSERC Discovery Grant (RGPIN-2016-04331) and a Sloan Research Fellowship. OH has been partially supported by a Koret Foundation early career scholar award. BW has been partially supported by NSF grant DMS-1711293. We thank the referee for very detailed comments.\\

\bigskip

\section{Preliminaries and Notation}\label{sec_prelim}

Standard references for varifolds are \cite{Allard,Simon_GMT}. For a more gentle introduction we recommend \cite{DeLellis}. Here, we briefly collect what we need for the present paper.\\

In the following, we denote by $\mathcal{H}^n$ the $n$-dimensional Hausdorff measure of subsets of $\mathbb{R}^{n+1}$. Recall that a set $M\subset \mathbb{R}^{n+1}$ is called \emph{locally $n$-rectifiable} if it has locally finite $\mathcal{H}^n$-measure and is contained in a countable union of $C^1$-hypersurfaces, up to a set of $\mathcal{H}^n$-measure zero.

\newcommand{\Ll}{\mathcal{L}}

\begin{definition}[integral varifold]
An \emph{integral $n$-varifold $V$ in $\mathbb{R}^{n+1}$} is an equivalence class of pairs $(M,\theta)$, where
$\theta(\cdot)$ is an nonnegative function in $\Ll^1_{loc}(d\Hh^n)$
such that $\theta(\cdot)$ is integer-valued $\Hh^n$-almost everywhere and
     such that
\begin{equation}
     M: = \{x: \theta(x)>0\}.
\end{equation}
   is countably $n$-rectifiable.
Two such pairs $(M_1,\theta_1)$ and $(M_2,\theta_2)$ are equivalent provided
$\theta_1=\theta_2$ except on a set of $\Hh^n$-measure $0$.
The varifold $V$ determines a Radon measure $\mu_V$ on $\RR^{n+1}$ given by
\begin{equation}
   \mu_V(S) = \int_{S\cap M}\theta(x)\,d\Hh^n(x).
\end{equation}
The \emph{support} of $V$ is defined to be the support of the measure $\mu_V$.
\end{definition}

In general, the $M$ and $\theta$ are only defined modulo $\Hh^m$-nullsets.
However, the varifolds that arise in this paper are either stationary or 
have constant mean curvature, and consequently the limit
\begin{equation}\label{density-exists}
  \theta_V(x) = \lim_{r\to 0}\frac{\mu_V(\mathbf{B}(x,r))}{\omega_nr^n}
\end{equation}
exists for all $x$ and is greater than $0$ at every point in the support of $\mu_V$. 
Thus for such $V$ there is a canonical choice of $M$ and $\theta$, 
namely $M=\spt\mu_V$ and $\theta=\theta_V$.

Also, the varifolds in this paper all have the property that $\theta=1$
for $\Hh^n$-almost every $x\in M=\spt\mu_V$.  Thus for such $V$, 
the varifold is determined by its support $M$ via
\begin{equation}
   \mu_V(S) = \Hh^n(S\cap M)
\end{equation}
and therefore $\theta=\theta_V$ is also determined by $M$ by~\eqref{density-exists}.

If $M$, $p$, and $S$ are a varifold, a point, and a set in $\RR^{n+1}$, 
 we will usually abbreviate expressions such as
$p\in \spt M$ and $\spt M\subseteq S$ by $p\in M$ and $M\subseteq S$.

We will often hide the multiplicity function $\theta$ in the notation, and simply talk about the varifold $M$.  

 The \emph{first variation} of $V$ is given by\footnote{Recall that any locally $n$-rectifiable $M\subset\mathbb{R}^{n+1}$ has an approximate tangent plane $T_xM$ at almost every $x\in M$.}
\begin{equation}
\delta V(X)=\int \mathrm{div}_{T_xM}X \, d\mu_M,\qquad \qquad \textrm{where } X\in C_c^1(\mathbb{R}^{n+1},\mathbb{R}^{n+1}),
\end{equation}
which generalizes the usual formula for the first variation of area of smooth hypersurfaces.\\

The varifold $V$ is called \emph{stationary}, if $\delta V=0$.
More generally, we say that the $V$ has \emph{constant mean curvature $h$} in an open set $U$
if there is a constant $h\ne 0$ and a $\mu_V$-measurable vectorfield $H$ with 
$|H(\cdot)|=h$ 
such that
\begin{equation}
  \delta V(X) = -\int X\cdot H\,d\mu_V
\end{equation}
for all $C^1$ vectorfields $X$ compactly supported in $U$.

By the monotonicity formula for stationary varifolds the function
\begin{equation}\label{monotonicity}
r\mapsto \frac{\mu_M(B(x,r))}{\omega_n r^n}\quad \textrm{ is monotone},
\end{equation}
and constant only on cones. The case of varifolds with constant mean curvature $h$
is similar: the function in~\eqref{monotonicity} need not be monotone, but if we multiply by $e^{hr}$, the resulting function is monotone.
In any case,  every $x\in M$ has a well defined \emph{density} given by \eqref{density-exists}.
By Allard's regularity theorem~\cite{Allard}, there exists some universal constant $\eps=\eps(n)>0$ so that every point $x\in M$ with $\theta_V(x)<1+\eps$ is regular. In particular, if we decompose
\begin{equation}
M=\reg M \cup \sing M,
\end{equation}
then the regular part is open and the singular part is closed.\\

By Allard's compactness theorem \cite{Allard}, any sequence $V_i=(M_i,\theta_i)$ of stationary or CMC integral $n$-varifolds with locally uniformly bounded measure and uniformly bounded mean curvature has a subsequence that converges to a stationary or CMC integral $n$-varifold $V=(M,\theta)$. Here, convergence in the sense of varifolds means that\footnote{In particular, this implies that $\mu_{M_i}$ converges to $\mu_M$ in the sense of measures. But the notion of convergence of varifolds also captures some important additional information about the convergence of the approximate tangent planes.}
\begin{equation}
\int \varphi(x,T_xM_i)\, d\mu_{M_i}(x)\to \int \varphi(x,T_xM)\, d\mu_{M}(x)
\end{equation}
for all compactly supported continuous functions $\varphi$ on the Grassmannian $\mathrm{Gr}_n(\mathbb{R}^{n+1})$.\\

Given $V=(M,\theta)$, a point $x\in M$, and scales $\lambda_i\to \infty$, let $V_i=\lambda_i\cdot(V-x)$ be the sequence of varifolds that is obtained from $V$ by shifting $x$ to the origin and rescaling by $\lambda_i$. By the monotonicity formula and Allard's compactness theorem, we can always pass to a subsequential limit $C$, called a \emph{tangent cone} at $x$. Tangent cones are always conical, i.e. $\lambda\cdot C=C$ for all $\lambda>0$. In particular, if some tangent cone at $x$ is a plane with multiplicity one (here, and in related situations, we also used that the multiplicity is simply a constant by Allard's constancy theorem~\cite{Allard}), then $x$ is regular by Allard's Regularity Theorem.

\bigskip

\section{Hopf lemma for varifolds}\label{sec_hopf}

The goal of this section is to prove the Hopf Lemma for varifolds (Theorem \ref{mirror-theorem_intro}), which we restate here for convenience of the reader:

\begin{theorem}[Hopf lemma for varifolds]\label{mirror-theorem_restated}
Let $M_1$ and $M_2$ be integral $n$-varifolds in $B(0,r)\subset\RR^{n+1}$ that 
have the same constant mean curvature $h$ (possibly $0$), and let
$\mathbb{H}$ be an open halfspace of $\RR^{n+1}$ with $0\in\partial\mathbb{H}$.
If
\begin{enumerate}[\upshape (i)]
\item\label{hopf-hypothesis-i}  $0\in M_1\cap M_2$ is a tame interior point for both $M_1$ and $M_2$,
\item\label{hopf-hypothesis-ii} $\partial \mathbb{H}$ is \textbf{not} the tangent cone to either $M_1$ or $M_2$ at $0$,
\item\label{hopf-hypothesis-iii} and $\reg M_1\cap \mathbb{H}$ and $\reg M_2\cap \mathbb{H}$ are disjoint,~
\end{enumerate}
then either
\begin{enumerate}[(a)]
\item\label{hopf-a} $h\ne 0$ and $M_1$ and $M_2$ curve oppositely at $0$ (in the sense of Definition~\ref{def_opp_curv}), or
\item\label{hopf-b} $M_1$ and $M_2$ are smooth at $0$, with distinct tangents.
\end{enumerate}
\end{theorem}

\begin{corollary}\label{hopf-corollary}
Let $M_1$ and $M_2$ be integral $n$-varifolds in $B(0,r)\subset\RR^{n+1}$ that 
have the same constant mean curvature $h$ (possibly $0$). 
Suppose
\begin{enumerate}[\upshape (i)]
\item\label{cor-hyp1}  $0\in M_1\cap M_2$ is a tame interior point for both $M_1$ and $M_2$,
\item\label{cor-hyp2} $\reg M_1$ and $\reg M_2$ are disjoint.
\end{enumerate}
Then $h\ne 0$ and $M_1$ and $M_2$  curve oppositely at $0$.
\end{corollary}

\begin{proof}[Proof of Corollary~\ref{hopf-corollary}]
Since $\reg M_1$ and $\reg M_2$ are disjoint, $0$ belongs to at most one of those sets.
If $0\in \reg M_i$, we let $\mathbb{H}$ be an open halfspace 
(with $0\in \partial \mathbb{H}$) whose boundary is not equal
to $\Tan(M_i,0)$.  If $0$ is not in $\reg M_1$ or $\reg M_2$, we let $\mathbb{H}$ be any
open halfspace (with $0\in \partial \mathbb{H}$).  Since $M_1$ and $M_2$ cannot both be smooth at $0$,
it follows from Hopf Lemma~\ref{mirror-theorem_restated}
 that $h\ne 0$ and that $M_1$ and $M_2$ are oppositely curved at $0$.
\end{proof}

The precise definition of curving oppositely, based on mutual tangent cones, is as follows:

\begin{definition}[mutual tangent cones, curving oppositely]\label{def_opp_curv} Let $M_1$ and $M_2$ be varifolds with constant mean curvature $h$  in an open ball $B(0,r)$ such that $0$ is a tame interior point for both $M_1$ and $M_2$.
\begin{enumerate}
\item The collection of mutual tangent cones $\Cc$ at $0$ is the  collection of all pairs 
  $(\Sigma_1,\Sigma_2)$ such that there is a sequence $\lambda_j\to\infty$
for which $\lambda_j M_a$ converges to $\Sigma_a$ for $a=1,2$. 
\item We say $M_1$ and $M_2$ curve oppositely at $0$, if $h\ne 0$ 
and if $\Cc$ has an element of the form $(\Sigma,\Sigma)$ such that the orientations of $\reg \Sigma $ induced from the mean curvature vectors of $\reg M_1$ and $\reg M_2$ are inconsistent.     
\end{enumerate}
\end{definition}

\begin{remark}\label{connected-remark}
If $\Sigma$ is a stationary $n$-cone whose singular set has $\Hh^{n-1}$ measure $0$, then $\reg \Sigma$ is connected 
by Theorem~\ref{brendle-connected} below. (Theorem~\ref{brendle-connected} is about
shrinkers, but it applies to stationary cones since 
every stationary cone is a shrinker). 
Thus in Definition~\ref{def_opp_curv}, two orientations on $\reg\Sigma$ are inconsistent (i.e., opposite at some points) if and only if they are opposite everywhere.
\end{remark}

To prove Theorem \ref{mirror-theorem_restated}, we first recall two Bernstein-type theorems for varifold shrinkers in a halfspace
 that were proved in our prior work~\cite[Sec. 3.2]{CHHW}. These theorems generalize important results for smooth two-dimensional surfaces by Brendle \cite{Brendle}.

\begin{definition}[{varifold shrinker, \cite[Def. 3.5]{CHHW}}]\label{def_sing_shrinker}
A \emph{varifold shrinker} is an integral $n$-varifold $\Sigma$ in $\mathbb{R}^{n+1}$ with finite entropy that is stationary with respect to the functional
\begin{equation}
F[\Sigma]=\int e^{-|x|^2/4}d\mu_\Sigma\, .
\end{equation}
\end{definition}

Essentially, the Bernstein-type theorems say that the only varifold shrinkers that are stable in a halfspace (respectively don't intersect in a halfspace) are flat planes. The precise statements are as follows:\footnote{In \cite[Thm. 3.8]{CHHW} a sharper result has been established, but the simplified version here is sufficient for our purpose.}

\begin{theorem}[{First Bernstein-type theorem \cite[Thm. 3.8]{CHHW}}]\label{brendle-stable}
Let $\Sigma$ be an $n$-dimensional varifold shrinker
 with multiplicity one in $\mathbb{R}^{n+1}$ such
  that $\Hh^{n-1}(\sing \Sigma\cap \mathbb{H})=0$.
 If $\reg \Sigma\cap \mathbb{H}$ is stable for the $F$-functional, then $\Sigma$ is a multiplicity one hyperplane.
\end{theorem}

\begin{theorem}[{Second Bernstein-type theorem \cite[Thm. 3.11, Cor. 3.12, Cor. 3.14]{CHHW}}]
\label{brendle-connected}
For $a=1,2$ let $\Sigma_a$ be the support of an $n$-dimensional varifold shrinker
in $\mathbb{R}^{n+1}$ with $\Hh^{n-1}(\sing \Sigma_a)=0$. 
Then $\reg\Sigma_a$ is connected and $\reg \Sigma_a\cap \mathbb{H}$
is connected. 
Furthermore, if $\reg \Sigma_1\cap \mathbb{H}$ and $\reg \Sigma_2\cap \mathbb{H}$ are nonempty and do not
intersect transversely at any point, then either
\begin{enumerate}
\item $\Sigma_1=\Sigma_2$, or
\item $\Sigma_1$ and $\Sigma_2$ are flat planes.
\end{enumerate}
\end{theorem}

We can now prove the main theorem of this section:

\begin{proof}[{Proof of Theorem \ref{mirror-theorem_restated}}] Consider the collection $\Cc$ of mutual tangent cones at the origin (see Definition \ref{def_opp_curv}), and suppose that $(\Sigma_1,\Sigma_2)\in \Cc$.  

\begin{claim*} $(\reg \Sigma_i) \cap \mathbb{H}$ is nonempty.
\end{claim*}

\newcommand{\BB}{\mathbf{B}}

\begin{proof}[Proof of claim]
Suppose, to the contrary, that $(\reg\Sigma_i)\cap \mathbb{H}$ is empty.
Since $\reg\Sigma_i$ is dense in $\Sigma_i$ (by tameness), it follows
that 
\begin{equation}\label{emptiness}
\Sigma_i\cap \mathbb{H} = \emptyset.
\end{equation} 
Let $v$ be the unit normal to $\partial \mathbb{H}$ pointing into $\mathbb{H}$.
Since $V_i$ is a minimal cone with multiplicity $1$ almost everywhere, 
we have, by the divergence theorem,
\begin{equation}\label{zephyr}
\int_{x\in \Sigma_i\cap \partial \BB} x\cdot v\, d\Hh^{n-1}
=
\int_{\Sigma_i\cap \BB} H\cdot v \,d\Hh^n
+ 
\int_{\Sigma_i\cap \BB} \Div_{\Sigma_i}v \,d\Hh^n 
= 0,
\end{equation}
where $\BB$ is the unit ball in $\RR^{n+1}$.
By~\eqref{emptiness}, 
$x\cdot v \le 0$ for
all $x\in \Sigma_i$.  
Thus by~\eqref{zephyr}, $x\cdot v= 0$ for $\Hh^{n-1}$-almost every 
  $x\in \Sigma_i\cap\partial \BB$.
The regular points of $\Sigma_i$ are dense in $\Sigma_i$, so $x\cdot v=0$
everywhere in $\Sigma_i\cap\BB$.
Thus $\Sigma_i$ is the plane $\partial \mathbb{H}$
with some multiplicity.  By tameness, the multiplicity is $1$.
By the Allard Regularity Theorem, $0$ is a regular point of $\Sigma_i$, and thus
$\partial \mathbb{H}$ is the unique tangent plane to $\Sigma_i$ at $0$.
But this violates Hypothesis~\ref{hopf-hypothesis-ii} of the theorem, and thus
the claim is proved.
\end{proof}

Observe that
\begin{equation}\label{transverse}
\text{$\reg \Sigma_1\cap \mathbb{H}$ and $\reg \Sigma_2\cap \mathbb{H}$ do not
intersect transversely at any point},
\end{equation}
since otherwise $\reg M_1\cap\mathbb{H}$ and $\reg M_2\cap\mathbb{H}$ would intersect,
 contrary to Hypothesis~\ref{hopf-hypothesis-iii}. 
 Thus, by Theorem~\ref{brendle-connected} (Second Bernstein-type theorem), either $\Sigma_1$ and $\Sigma_2$ are distinct planes
or $\Sigma_1=\Sigma_2$.  In the first case, we are done, so (in our argument by contradiction) we can assume from now on that
\begin{equation}\label{same}
\Sigma_1=\Sigma_2 \quad\text{for all $(\Sigma_1,\Sigma_2)\in\Cc$}.
\end{equation}
If $h\ne 0$ and if there is any $(\Sigma, \Sigma)\in \Cc$ for which the orientations
on $\Sigma$ are inconsistent, then we are done.   
Thus (to prove the theorem by contradiction) 
 from now on we  assume that (i) $h=0$, or (ii)  $h\ne 0$
and the orientations on $\Sigma$ are consistent for every $(\Sigma, \Sigma)\in \Cc$.

Suppose that $(\Sigma,\Sigma)\in\Cc$. 
We now apply Theorem~\ref{brendle-connected} 
to the shrinkers $\Sigma$ and $\partial B(0,\sqrt{2n})$.
They are not equal since $0$ is in $\Sigma$ but not in $\partial B(0,\sqrt{2n})$.
Also, $\partial B(0,\sqrt{2n})$ is not planar.
Thus by Theorem~\ref{brendle-connected},
there is a point where $\reg\Sigma\cap\mathbb{H}$ and 
 $\partial B(0,\sqrt{2n})$ intersect transversely.
 Consequently,
\begin{equation}\label{2-ball}
    \reg \Sigma \cap B(0,\sqrt{2n}) \cap \mathbb{H}\ne \emptyset.
\end{equation}

We now prove existence of a nontrivial Jacobi field. To this end, we will first locate regions of approximately maximal regularity scale. 
For any set $S\subseteq\mathbb{H}$ and any $p\in S$, denote by $R(S,p)$ be the regularity scale of $S$ at $p$ in $\mathbb{H}$, i. e., the supremum of $r\geq 0$ such that 
\begin{enumerate}
\item $B(p,r)\subseteq \mathbb{H}$, 
\item $S\cap B(p,r)$ is a smooth $n$-dimensional manifold\ (without boundary in $B(p,r)$) properly
embedded in $B(p,r)$, and
\item the norm of the second fundamental form at each point of $S\cap B(p,r)$ is $\le 1/r$.
\end{enumerate}
 We consider the quantity
\begin{equation}
   \rho(S) := \{ \sup R(S\cap\mathbb{H} , p):  p\in S\cap\mathbb{H}, \, |p|\le \sqrt{2n}\}.
\end{equation}

\begin{claim}[lower bound for regularity scale]\label{claim_reg} For $a=1,2$ we have
\begin{equation}
\eta_a:=\liminf_{\lambda \to\infty} \rho(\lambda M_a)>0.
\end{equation}
\end{claim}

\begin{proof}[{Proof of Claim \ref{claim_reg}}]
Choose $\lambda_j\to\infty$ so that $\rho(\lambda_j M_1)\to \eta_1$.
By passing to a subsequence, we can assume that $\lambda_j M_1$ converges to a limit $\Sigma$.
By~\eqref{2-ball},
\begin{equation}
   \rho(\Sigma)>0.
\end{equation}
By smooth convergence of $\lambda_jM_1$ to $\Sigma$ at the regular points, which follows from Allard's regularity theorem \cite{Allard}, we get
\begin{equation}
   \lim_{j\to\infty} \rho(\lambda_j M_1) = \rho(\Sigma).
\end{equation}
Thus $\eta_1>0$. Likewise, $\eta_2>0$. This proves the claim.
\end{proof}

Continuing the proof of the theorem, set\footnote{Using~\eqref{same}, it is not hard to show that $\eta_1=\eta_2$, but we do not need that fact.}
\begin{equation}
  \eta:=\min\{\eta_1,\eta_2\}.
\end{equation}

Choose $\Lambda<\infty$ so that $\rho(\lambda M_a) > \eta/2$ for all $\lambda \ge \Lambda$.

For $\lambda \ge \Lambda$, consider the quantity
\begin{equation}
    \psi(\lambda) := \sup \left\{ \dist(x, \lambda M_2): \text{$x\in  \lambda M_1\cap\mathbb{H}$, $|x|\le \sqrt{2n}$, and $R( \lambda M_1\cap\mathbb{H}, x)\ge \eta/2$} \right\}.
\end{equation}
Note that the supremum will be attained at some (not necessarily unique) point  $x_\lambda$.
By~\eqref{same} we have
\begin{equation}
\lim_{\lambda \to\infty}\psi(\lambda)=0.
\end{equation}
Choose $\lambda_j\ge j$ so that
\begin{equation}\label{eq_point_selection}
     \psi(\lambda_j) \ge (1-j^{-1}) \sup_{\lambda \ge \lambda_j} \psi(\lambda).
\end{equation}

By passing to a subsequence, we can assume that $\lambda_jM_1$ converges to a limit $\Sigma$ 
and that $x_{\lambda_j}$ converges to a point $x\in  \reg \Sigma\cap\mathbb{H}$ (in particular, observe that by definition of the regularity scale the point $x$ is at definite distance from $\partial\mathbb{H}$).  
Again by~\eqref{same}, the blowup sequence $\lambda_j M_2$ also converges to $\Sigma$. Note that the convergence is smooth on compact subsets of $\mathbb{R}^{n+1}\setminus\sing \Sigma$.

Choose open subsets $\Omega_j$ of $\mathbb{H} \cap\reg \Sigma$ such that
\begin{gather*}
x\in \Omega_1 \subset \Omega_2 \subset \dots, \\ 
\Omega_j\subset\subset \reg \Sigma\cap\mathbb{H}, \\
\cup_j\Omega_j=\reg \Sigma\cap\mathbb{H}.
\end{gather*}
Since $\reg \Sigma\cap \mathbb{H}$ is connected
 (by Theorem~\ref{brendle-connected} and Remark~\ref{connected-remark}), we can choose the $\Omega_j$ to be connected.

Let $\nu$ be a unit normal vectorfield on $\reg \Sigma$.  
By the smooth convergence of $\lambda_jM_a$ on the regular portion of $\Sigma$, 
we can, by passing to a subsequence, assume that for $a=1,2$
there are functions 
\begin{equation}
  u^a_j: \Omega_j   \to \mathbb{R}
\end{equation}
such that
\begin{align}
   \{ p + u^a_j(p)\nu(p): p\in \Omega_j\} &\subset \lambda_jM_a,
\end{align}
and such that $u^a_j$ converges to $0$ smoothly on compact
subsets of 
\begin{equation}
 \reg \Sigma\cap \mathbb{H}.
\end{equation}

By relabelling, we may assume that $u^2_j>u^1_j$ on $\Omega_j$. Then, by the Harnack inequality,  and since $M_1$ and $M_2$ are consistently oriented at $0$ in the CMC case, the renormalized sequence
\begin{equation}
   \frac{u^2_j - u^1_j}{ u^2_j(x_{\lambda_j}) - u^1_j(x_{\lambda_j})}
\end{equation}
converges smoothly (perhaps after passing to a further subsequence) to a positive solution
\begin{equation}
    u:  \reg \Sigma \cap \mathbb{H} \to \mathbb{R}_+
\end{equation}
of the linearization of the minimal surface equation
\begin{equation}
\Delta u +|A|^2u=0.
\end{equation}

Moreover, by construction,
\begin{equation}
   u(x)=1,
\end{equation}
and, thanks to \eqref{eq_point_selection}, we have
\begin{equation}\label{growth_cone}
  |u(\lambda p)|\le \lambda \quad\text{for all $\lambda \in (0,1)$ and all $p\in B(0,\sqrt{2n})\cap \mathbb{H}$ with $R( \Sigma\cap\mathbb{H},p)> \eta/2$}.
\end{equation}

\begin{claim}[stability]\label{claim_stability} $ \reg \Sigma\cap \mathbb{H}$ is a stable critical point of the $F$-functional.
\end{claim}

\begin{proof}[{Proof of Claim \ref{claim_stability}}]
Since $\Sigma$ is a minimal cone, it is a critical point of the area functional as well as a critical point of the $F$-functional. Regarding the second variation, recall that the Jacobi operator for the area functional is
\begin{equation}
L=\Delta + |A|^2,
\end{equation}
and the Jacobi operator for the $F$-functional is
\begin{equation}
\mathcal{L}=\Delta + |A|^2-\tfrac{1}{2}x^\top \cdot \nabla + \tfrac{1}{2}.
\end{equation}
The existence of a positive function $u$ in the kernel of $L$ immediately implies stability with respect to the area functional, see e.g. \cite{FischerColbrieSchoen}, but does not directly yield stability with respect to the $F$-functional. To establish the latter, we consider the function 
\begin{equation}
w(x,t)=e^{t/2}u(e^{-t/2}x),
\end{equation}
where $x\in \reg\Sigma\cap\mathbb{H}$ and $t\geq 0$. Note that this is well defined, since $\reg\Sigma\cap\mathbb{H}$ is a cone. The function $w$ satisfies
\begin{equation}
Lw=0,
\end{equation}
and
\begin{equation}
\partial_t w =\tfrac{1}{2}w -\tfrac{1}{2}x^\top \cdot \nabla w,
\end{equation}
hence in particular
\begin{equation}
\partial _{t} w=\mathcal{L}w.
\end{equation}
Moreover, the growth condition \eqref{growth_cone} implies that 
\begin{equation}\label{boundforw}
  |w(x,t)|\le 1 \quad\text{for all $t\ge 0$ and all $x\in B(0,\sqrt{2n})\cap \mathbb{H}$ with $R( \Sigma\cap \mathbb{H},x)> \eta/2$}.
\end{equation}
By the parabolic Harnack inequality this implies 
\begin{equation}\label{harnack-bound}
    \sup_{ K\times [1,\infty)} w < \infty \quad\text{for $K\subset\subset  \reg \Sigma\cap \mathbb{H}$}.
\end{equation}

Let $U\subset\subset \reg\Sigma\cap \mathbb{H}$ be a connected open set. Let $\lambda$ be the first Dirichlet eigenvalue of $\mathcal{L}$ on $U$, and let $f$ be the corresponding eigenfunction.
Then $f$ is nonzero at all points of $U$. By multiplying by a constant, we can assume that $0< f\le w(\cdot,0)$.
Thus, by the maximum principle, we get
\begin{equation}
    e^{-\lambda t} f(\cdot) \le w(\cdot,t) 
\end{equation}
for all $t\ge 0$. 
Together with \eqref{harnack-bound} this yields that $\lambda\geq 0$. Since $U$ was arbitrary, this proves the claim.
\end{proof}

Thus, by Theorem~\ref{brendle-stable} (First Bernstein-type theorem), $\Sigma$ is a flat plane. Together with the tameness assumption and Allard's regularity theorem~\cite{Allard}, this implies that $M_1$ and $M_2$ are smooth at $0$. 
Hence, the classical Hopf lemma in the smooth setting gives $\Sigma_1\neq \Sigma_2$, contradicting our assumption~\eqref{same}. 
This concludes the proof of the theorem.
\end{proof}

\bigskip

\section{Moving planes for compact varifolds with smooth boundary}\label{sec_boundary}

In this section, we implement our version of the moving plane method in the setting of compact varifolds with smooth boundary. Theorem \ref{main_boundary_theorem} and Theorem \ref{main_boundary_CMC_theorem} are proved simultaneously in Theorem \ref{main_boundary_combined} and Corollary \ref{full_sym_cor} below.

\begin{theorem}[compact varifolds with smooth boundary]\label{main_boundary_combined}
Suppose $\Gamma=\Gamma_1\cup\Gamma_2$, where $\Gamma_1$
 and $\Gamma_2$
are smooth, closed, strictly convex $(n-1)$-dimensional surfaces in parallel $n$-planes $Q_1$ and $Q_2$
 in $\RR^{n+1}$.
Suppose that $M$ is a compact connected integral $n$-varifold that has constant
mean curvature $h$ in $\RR^{n+1}\setminus(\Gamma_1\cup\Gamma_2)$,
and suppose that 
the density of $M$ is $1/2$ at each point of $\Gamma$.
In case $h=0$ (the stationary case), we assume that  $M$ has no triple junctions.
In case $h\ne 0$, we assume that $M$ is tame, lies in the slab bounded by $Q_1\cup Q_2$ and is nowhere tangent to the boundary of the slab.

Suppose $P$ is a hyperplane of symmetry of the boundary $\Gamma$ and
that $P$ is perpendicular to the planes $Q_i$. 
Then $P$ is a plane of symmetry of $M$. Moreover, the portion of $M$ on each side of $P$ is a smooth graph over a region in $P$. 
\end{theorem} 

\begin{remark}  A careful inspection of the proof shows that the assumption that $\Gamma_i$ is smooth, closed 
and strictly convex can be replaced by the weaker assumption that $\Gamma_i$ is smooth, closed, connected and each line perpendicular to $P$ intersects $\Gamma_i$ in at most 2 points.
\end{remark}

Before starting the moving plane argument, let us record the following basic properties of $M$:

\begin{proposition}[enclosed domain and mean curvature]\label{interior_lemma}
Let $D_1$, $D_2$ be the convex domains in $Q_1$, $Q_2$ bounded by $\Gamma_1$ and $\Gamma_2$. Then 
\begin{enumerate}
\item $M\cup D_1\cup D_2$ bounds a compact set $K \subseteq \mathbb{R}^{n+1}$. 
\item $\reg M$ is connected.
\item If $h\ne 0$, then either $\mathbf{H}=h\nu$ everywhere in $\reg M$,
or $\mathbf{H}= - h\nu$ everywhere in $\reg M$, where $\nu$ the outward unit normal
to $K$.
\end{enumerate}
\end{proposition}

\begin{proof}
By Lemma~\ref{lemma_tame} below in $h=0$ case and by assumption in the $h\ne 0$ case, the varifold $M$ is tame. 

Hence, by Definition \ref{def_tame} (tameness) any tangent cone to $M$ that splits off an $\mathbb{R}^{n-1}$-factor must be a multiplicity-one hyperplane. By standard stratification (see e.g. \cite{Simon_GMT}) this implies
that the singular set has Hausdorff dimension at most $n-2$, which in turn implies
that the singular set has 
  $\mathcal{H}^{n-1}$ measure $0$ and that $\reg M$ is dense in $M$.
Since $\sing M\subset \mathbb{R}^{n+1}$ has codimension bigger than $2$, we have the following two properties:
\begin{itemize}
\item any smooth curve in $\mathbb{R}^{n+1}$ can be perturbed such that it avoids $\sing M$,
\item any smooth $2$-disc in $\mathbb{R}^{n+1}$ bounding a closed curve can be perturbed such that it avoids $\sing M$ and meets $\reg M$ transversally.
\end{itemize}
Hence, by standard intersection theory (see e.g. \cite{Samelson}) the set $M\cup D_1\cup D_2$ encloses a compact set $K \subseteq \mathbb{R}^{n+1}$:
\begin{equation}
\partial K=M\cup D_1\cup D_2.
\end{equation}

To prove connectedness of $\reg M$, suppose that $\reg M$ is not connected.  Let $M'$ be a connected component of $\reg M$.
Then $M'':= (\reg M)\setminus M'$ is nonempty.
Since $M$ is connected and since $\reg M$ is dense in $M$, 
there is a point $p$ that lies in the closures of $M'$ and of $M''$.
By translating, we can assume that $p=0$.
Tameness implies that the multiplicity-one varifolds $V'$ and $V''$ associated to $M'$ and $M''$ have constant mean curvature $h$.
(The corresponding fact for shrinkers is proved in~\cite{CHHW}: see the proof
of Assertion~(5) of Lemma~3.7, beginning with ``Alternatively". 
  Exactly the same proof works in this setting.)
By Corollary~\ref{hopf-corollary}, 
$M_1$ and $M_2$ are oppositely curved at $0$ and thus have
a mutual tangent cone $(\Sigma, \Sigma)$.
The corresponding tangent cone of $M$ is $\Sigma$ with multiplicity $2$ almost everywhere,
violating tameness.  This completes the proof that $\reg M$ is connected.

The last assertion of the proposition follows immediately from the connectedness of
 $\reg M$.
\end{proof}

After these preparations, we can now implement our new moving plane method:

\begin{proof}[Proof of Theorem \ref{main_boundary_combined}]
We can assume without loss of generality that the planes $Q_1$ and $Q_2$ 
are parallel to the plane $\{x_{n+1}=0\}$ and that $P$ is the plane $\{x_1=0\}$.
If $Z$ is any set (such as $M$ or $\Gamma$) in $\RR^{n+1}$ and if $s\ge 0$, let 
\begin{align*}
  Z_s^+ &= Z \cap \{x_1>s\},\\
  Z_s^- &= Z\cap \{x_1<s\}, 
\end{align*}
and let $Z_s^*$ be the image of $Z_s^+$ under reflection in the plane $\{x_1=s\}$.

\newcommand{\Ss}{\mathcal{S}}
Let $\Ss$ be the set of $s>0$ such that
\begin{enumerate}[\upshape (i)]
\item\label{regular-item} In $M\cap \{x_1\ge s\}$, every point is a regular point and $\ee_1\cdot\nu>0$, where $\nu$ is the unit normal to $M$ that points out of $K$.
\item Each line parallel to the $x_1$-axis intersects $M_s^+$ in at most one point.
\item $K_s^* \subseteq K_s^-$.
\item\label{disjoint-item} $M_s^*$ and $M_s^-$ are disjoint.
\end{enumerate}
Note that $\Ss$ is open and nonempty (since it contains large values of $s$), and that if $s\in \Ss$, then every $s'>s$ is also in $\Ss$.
Thus
\begin{equation}
  \Ss = (t,\infty),
\end{equation}
where $t:=\inf \Ss\ge 0$. 

From~\ref{regular-item}--\ref{disjoint-item}, we see that
\begin{enumerate}[\upshape (i)$^\prime$]
\item\label{graphy-prime} In $M_t^+$, every point is a regular point and $\ee_1\cdot\nu>0$.
\item\label{normal-prime} Each line parallel to the $x_1$-axis intersects $M_t^+$ in at  most one point.
\item $K_t^* \subseteq K_t^-$.
\end{enumerate}

\begin{claim} $t=0$.
\end{claim}

\begin{proof}[Proof of the claim] If $M_t^*$ touched $M_t^-$ at some point $q$, then by the Solomon-White maximum principle (Theorem \ref{strong_max_varifold_intro}), the entire connected component of $M_t^*$ containing $q$
would lie in $M_t^-$.  
Let $U$ be the largest open subset of $\reg M$ such that the image of $U$ under
reflection in $\{x_1=t\}$ is contained in $M$.  
By unique continuation and by connectedness of $\reg M$, 
it follows that $U=\reg M$ and thus that the plane $\{x_1=t\}$ is a plane of symmetry of $M$.
Therefore it is a plane of symmetry of $\Gamma$, and hence $t=0$ as claimed.

Thus in proving the claim, we may assume that $M_t^*$ and $M_t^-$ are disjoint. 

Suppose that the claim is false, i.e., that $t>0$. Note that $\ee_1\cdot\nu>0$ at each regular point $q$ of $(M-\Gamma)\cap\{x_1=t\}$. Indeed, if $M_{t}^{\ast}\neq \emptyset$ this follows from the smooth Hopf lemma, and if $M_{t}^{\ast}=\emptyset$ we clearly have $\ee_1 \cdot \nu=1$. Now, for any regular boundary  point $q\in M\cap \Gamma\cap \{x_1=t\}$,  our assumption of no tangency of $M$ to $Q_i$ when $h\neq 0$, and the standard smooth Hopf lemma (w.r.t. the planes $Q_i$) when $h=0$ imply that $\nu \neq \pm e_{n+1}$. Since the normal $N$ to $\Gamma$ within $Q_1\cup Q_2$ satisfies $\ee_1\cdot N \neq 0$, this and continuity imply that $\ee_1 \cdot \nu >0$. Thus, since $t\notin \Ss$, there must be a singular point $p$ in $M\cap\{x_1=t\}$.
In light of our assumption of density $1/2$ at boundary points of $M$, the Allard Boundary Regularity Theorem~\cite{Allard_boundary} implies that $p$ is not in $\Gamma$. Furthermore, this also shows that $M_t^{\ast}\neq \emptyset$ as otherwise all points in $(M-\Gamma)\cap\{x_1=t\}$ would be regular by tameness and \cite{SolomonWhite}.

Let $K'$ and $M'$ be the images of $K$ and $M$ under reflection in the plane $\{x_1=t\}$. Since $p$ is a singular point for $M$, no tangent cone to either  $M$ or $M'$ is $\{x_1=0\}$. 
Thus, by the  Hopf lemma for varifolds (Theorem \ref{mirror-theorem_intro})  (applied
to $M$, $M'$ and the halfspace $\{x_1>t\}$ at the point $p$), either
\begin{enumerate}[\upshape (a)]
\item\label{opposites} $h\ne 0$ and $M$ and $M'$ are curve oppositely at $p$, or
\item\label{crossing} $M$ and $M'$ are smooth at $p$, with distinct tangents.
\end{enumerate}
Now~\ref{crossing} cannot occur since $p$ is a singular point of $M$.
But~\ref{opposites} also cannot occur, since if $h\ne 0$ and if $(\Sigma,\Sigma)$ is a mutual
tangent cone pair to $M$ and $M'$, then either $\mathbf{H}\equiv h\nu$ everywhere
on $\reg M$ or 
or $\mathbf{H}\equiv -h \nu$ everywhere on $\reg M$
(by Proposition \ref{interior_lemma}), which forces the orientations on $\Sigma$ to be consistent and therefore that $M$ and $M'$ do not curve oppositely at $p$.
The contradiction completes the proof of that $t=0$.
\end{proof}

Continuing the proof of the theorem, since $t=0$, we see that
 $K_0^*\subseteq K_0^-$.
The same argument shows that the reflected image of $K\cap \{x_1<0\}$ lies in $K\cap \{x_1>0\}$.
Thus $K$ and $M$ are invariant under reflection in $\{x_1=0\}$. 
This together with~\ref{graphy-prime} and~\ref{normal-prime} completes the proof of~Theorem \ref{main_boundary_combined}.
\end{proof}

\begin{corollary}\label{full_sym_cor}
Suppose in Theorem \ref{main_boundary_combined} that $\Gamma_1$ and $\Gamma_2$ are $(n-1)$-spheres in parallel $n$-planes  that are rotationally invariant about the same axis.
   Then $M$ is smooth and rotationally invariant about that axis, and thus is a portion of an $n$-dimensional catenoid or of a Delaunay hypersurface.
\end{corollary}

\begin{proof}
We can apply Theorem \ref{main_boundary_combined} for each $n$-plane $P$ containing the axis.
It follows that $M$ is rotationally invariant about the axis, and smooth except possibly along the axis.
If $M$ contained a point on the axis, then the tangent cone would be rotationally invariant about the axis,
and thus by tameness would be a multiplicity-one plane. By Allard's Regularity Theorem, it
would be a regular point.   Thus $M$ is smooth everywhere. This proves the corollary.
\end{proof}

Let us observe that the above argument also yields uniqueness of CMC-spherical caps:

\begin{proof}[Proof of Theorem \ref{CMC_uniq}]
The proof is similar as above, with the only change that now the boundary has only one component instead of two components.
\end{proof}

The following lemma has been used in the above proof:

\begin{lemma}[tameness]\label{lemma_tame} Every stationary integral $n$-varifold $M$ satisfying the assumptions of Theorem \ref{main_boundary_combined} is tame.
\end{lemma}

\begin{proof}
Let $x\in M\setminus(Q_1\cup Q_2)$. For $a=1,2$ denote by $E_a$ the exterior cone over $\Gamma_a$ with vertex $x$:
\begin{equation}
E_a=\{x+ \lambda(y-x): y\in \Gamma_a\}.
\end{equation}
Applying the extended monotonicity formula from \cite{EkholmWhiteWienholtz} to $M\cup E_1\cup E_2$, we get
\begin{equation}\label{density_bound}
\Theta(M,x)<2,
\end{equation}
where we used that $E_a$ has density at infinity strictly less than $1$, thanks to convexity. This shows that all tangent cones of $M$ have multiplicity $1$. 

The assumption of Theorem \ref{main_boundary_combined} in case $h=0$ is that $M$ has no triple-junction tangents. To obtain tameness, we need to show that every \textit{tangent cone of $M$} has no triple-junction tangent either. Suppose towards a contradiction that $M$ has a tangent cone $C$ at some point $x$, such that $C$ itself had a triple-junction tangent cone $T$ (not at the origin, of course). Then there exists $x_i\to x$ and $\lambda_i\to \infty$ such that $M_i:=\lambda_i(M-x_i)$ converges to $T$ in the sense of varifolds.  Let $\gamma$ be a circle of radius $1$ around the origin in the cross-sectional plane of the triple junction, and observe that $\gamma$ intersects $T$ transversally at three regular points. Hence, by Allard's regularity theorem, $\gamma$ intersects $M_i$ transversally at three regular points, when $i$ is large enough. Thus, there exists a curve $\tilde{\gamma}$ intersecting $M$ transversally at three regular points, and nowhere else. 

On the other hand, as $M$ itself has no triple-junction singularity, it follows from dimension reduction that $\mathcal{H}^{n-1}(\sing M)=0$. In particular, every closed smooth curve in $\mathbb{R}^{n+1}-\sing M $ transversal to $\reg M$ must intersect $M$ an even number of times, a contradiction.
\end{proof}

To conclude this section, we give the proof of Corollary \ref{GMT_cor}:

\begin{proof}[Proof of Corollary \ref{GMT_cor}]
To see that~\ref{div-inequality} and~\ref{mod-2-boundary} imply~\ref{annulus-stationary-hypothesis}, \ref{annulus-density-hypothesis},
and~\ref{annulus-triple-hypothesis},
note that~\ref{div-inequality} implies stationarity of $M$ away from~$\Gamma$. 
By the convex hull property, any tangent cone to $M$ at any point of $\Gamma$ lies in a wedge.
By Corollary~32 of~\cite{white-boundary-mcf}, such a tangent cone is a union of halfplanes.  
The mod $2$ hypothesis~\ref{mod-2-boundary} implies that the number $k$ of halfplanes (counting multiplicity) is odd.
The inequality~\ref{div-inequality} then implies that $k=1$, and thus that the density is $1/2$.  (See Theorem~34 of \cite{white-boundary-mcf}
for details.)  Finally, \ref{mod-2-boundary} implies that there are no triple junctions.  (If there were a triple junction, we could find a
small circle that does not link $\Gamma$ and that intersects $ M$ transversely in exactly three points, each of multiplicity one.)
\end{proof}

\bigskip

\section{Moving planes for varifolds without boundary}\label{sec_complete}

In this section, we implement our version of the moving plane method in the setting of varifolds without boundary 
and we prove Theorem \ref{main_complete_theorem}, which we restate here for convenience of the reader:

\begin{theorem}[Uniqueness of the catenoid]\label{main_complete_theorem_restated}
Suppose $M$ is stationary integral $n$-varifold in $\mathbb{R}^{n+1}$ such that
\begin{enumerate}
\item some tangent cone at infinity is a multiplicity-2 plane,
\item $M$ has at least two ends,
\item $M$ has no triple junctions.
\end{enumerate}
Then $M$ is a smooth hypersurface of revolution, and thus either a pair of parallel planes or an $n$-dimensional catenoid.
\end{theorem}

\begin{proof}
First observe that assumptions (i) and (ii) together with monotonicity imply that
\begin{equation}
\Theta(M,x)<2
\end{equation}
for every $x\in M$. Thus, the same argument as in the last paragraph of the proof of Lemma \ref{lemma_tame} shows that $M$ is tame. If $n=2$, then $M$ is smooth and rotational symmetry follows from the classical result of Schoen \cite{Schoen}, so we can assume from now on that $n\geq 3$. We can also assume that $M$ is connected, since otherwise by monotonicity it must be the union of two parallel planes.

By rotating, we can assume that the horizontal plane with multiplicity $2$ is a tangent
cone at infinity to $M$.  It follows that the tangent cone is unique, and that
there is an $R<\infty$ for which $M\setminus B^n(0,R)$ consists
of the graphs of two smooth functions 
$u_\pm : \mathbb{R}^{n}\setminus B^n(0,R)\rightarrow \mathbb{R}$, satisfying
\begin{equation}
\mathrm{lim}_{|x|\rightarrow \infty} |\nabla u_\pm(x) |=0.
\end{equation} 
(For proof, 
see the discussion on pages 269 and 270 of~\cite{Simon-Ends}.  Specifically, choose $R>0$ large enough that $M\setminus B(0,R)$ has two components, and then apply 
that discussion to each of those components.)   

Hence, the graphical minimal surface equation outside a large ball is just a perturbation of Laplace's equation on $\mathbb{R}^n$, which has Green's function $c_n|x|^{2-n}$, and it easily follows that
\begin{equation}\label{expansion1}
u_\pm = a_\pm + O(|x|^{2-n}),
\end{equation}
see \cite[Proposition 3]{Schoen}.
Using this, the maximum principle for varifolds (Theorem \ref{strong_max_varifold_intro}) implies that
\begin{equation}\label{expansion2}
\textrm{$M$ must be contained in the slab $\{a_{-}<x_{n+1}<a_{+}\}$,}
\end{equation}
in particular $a_{-}<a_{+}$. After translating in $x_{n+1}$-direction we can assume that $a_{-}=-a_{+}$. Moreover, as in Proposition \ref{interior_lemma}, we see that $M$ encloses a domain $K$ with one end.\\

\begin{claim}[reflection symmetry and smoothness away from plane]\label{claim_refl}
Reflection across $\{x_{n+1}=0\}$ is a symmetry of $M$. Moreover, $M\cap \{ x_{n+1}\neq 0\}$ is smooth.
\end{claim}

\begin{proof}[Proof of Claim \ref{claim_refl}]

For $s\in (a_{-},a_{+})$ we let
\begin{align*}
M_s^{-}&=\{x\in M\; : \; x_{n+1}<s\},\\
M_s^+&=\{x\in M\; : \; x_{n+1}>s\},
\end{align*}  
and we let $M_s^*$ be the image of $M_s^+$ under reflection in the plane $\{x_{n+1}=s\}$. 
Similarly, we let
\begin{align*}
K_{s}^{-}&=\{x\in K\; : \; x_{{n+1}}<s\},\\
K_{s}^{+}&=\{x\in K\; : \; x_{n=1}> s\},
\end{align*}  
and we let $K_s^*$ be the image of $K_s^+$ under reflection in the plane $\{x_{n+1}=s\}$.

\newcommand{\Ss}{\mathcal{S}}
Let $\Ss$ be the set of $s>0$ such that
\begin{enumerate}[\upshape (i)]
\item\label{regular-item} In $M\cap \{x_{n+1}\ge s\}$, every point is a regular point and $\ee_{n+1}\cdot\nu>0$, where $\nu$ is the unit normal to $M$ that points out of $K$.
\item Each line parallel to the $x_{n+1}$-axis intersects $M_s^+$ in at most one point.
\item $K_s^* \subset K_s^-$.
\item\label{disjoint-item} $M_s^*$ and $M_s^-$ are disjoint.
\end{enumerate}
Note that if $s\in \Ss$, then every $s'\in [s,a_{+})$ is also in $\Ss$. By \eqref{expansion1} and \eqref{expansion2}, if $t$ is sufficiently close to $a_+$ then $t\in \mathcal{S}$ and all points in $M\cap \{x_{n+1}\geq t\}$ are smooth. This gets the moving plane method started. Moreover, using again  \eqref{expansion1} and \eqref{expansion2} we see that  $\Ss$ is open. Hence,
\begin{equation}
  \Ss = (t,\infty),
\end{equation}
where $t:=\inf \Ss\ge 0$. Suppose towards a contradiction that $t>0$.

Using the Solomon-White maximum principle (Theorem \ref{strong_max_varifold_intro}) and \eqref{expansion1} we see that $M_t^*$ and $M_t^-$ are disjoint. 
By the smooth Hopf lemma, $\ee_{n+1}\cdot\nu>0$ at each regular point of $M\cap\{x_{n+1}=t\}$.
Thus, since $t\notin \Ss$, there must be a singular point $p$ in $M\cap\{x_{n+1}=t\}$.
Let $M'$ be the images of $M$ under reflection across the plane $\{x_{n+1}=t\}$. Letting $\Sigma$ be any tangent cone to $M$ at $p$, applying the Hopf lemma for varifolds (Theorem \ref{mirror-theorem_intro}) to $M$, $M'$, the point $p$, and the half space $\{x_{n+1}<t\}$ gives that $\Sigma$ is a plane. This contradicts the fact that $p$ is a singular point and thus shows that $t=0$.
This implies the assertion.
\end{proof}

\begin{claim}[rotational symmetry and smoothness away from the axis]\label{claim_refl2}
$M$ is rotationally symmetric around the $x_{n+1}$-axis and smooth away from the $x_{n+1}$-axis.
\end{claim}

\begin{proof}[Proof of Claim \ref{claim_refl2}]
By rotation of coordinates it suffices to consider the moving plane $\{x_1=t\}$. The argument is similar as above, with the only difference that getting the moving plane method started requires a somewhat more careful expansion at infinity. To this end, recall that by \cite[Prop. 3]{Schoen} after a suitable shift in the $(x_1,\ldots,x_n)$-plane the function $u_{+}$ (and by the established reflection symmetry, also $-u_{-}$) can be expanded as 
\begin{equation}\label{uplus_bd}
u_{+}(x)= a-b |x|^{2-n}+O(|x|^{-n}),
\end{equation}
where $a>0$, and moreover
\begin{equation}\label{der_bd}
\partial_i u_{+}(x)=-b(2-n)\frac{x_i}{|x|^{n}}+O(|x|^{-n-1}).
\end{equation}
(see \cite[Proposition 3]{Schoen} and its proof, and the argument in the second paragraph of \cite[page 807]{Schoen}). By the maximum principle for varifolds (Theorem \ref{strong_max_varifold_intro}) we have $b\geq 0$. Integrating the equation $\Delta_{M}x_{n+1}=0$ over a large annulus, and applying the divergence theorem, yields $b>0$. Using this and \eqref{uplus_bd} we see that for $\eps>0$ small enough we have that 
$M\cap \{x_{n+1} =a-\eps\}$ is $O(R_{\eps}^{-1})$ close in the Hausdorff sense to the  round $(n-1)$-sphere of radius 
\begin{equation}
R_\eps = \left(\frac{b}{\eps}\right)^{\tfrac{1}{n-2}},
\end{equation}
centered at $(0,\ldots, 0,a-\eps)$ in the plane $\{x_{n+1} =a-\eps\}$. By \eqref{der_bd} we further get that $M\cap \{x_{n+1} =a-\eps\}$ is smooth, contains a single sheet, and its exterior normal within the plane $\{x_{n+1}=a-\varepsilon\}$ satisfies $n(x)=\frac{x}{|x|}+O(|x|^{-2})$. Thus, $M\cap \{x_{n+1} =a-\eps\}$  is a normal graph over the sphere of radius $R_\eps$, having a global $C^1$ norm smaller than $CR_{\eps}^{-1}$. The same is true for $M\cap \{x_{n+1}=-a+\eps\}$. This shows that there is no contact at infinity and gets the moving plane method started. Pushing the moving plane as in the proof of the previous claim, this yields the assertion.
\end{proof}

By the above two claims, $M$ is reflection symmetric and rotationally symmetric and smooth, except possibly at the origin. If the origin where contained in $M$, then by rotational symmetry and tameness the tangent cone there would be a multiplicity-one plane. Thus, $M$ is smooth, and hence an $n$-dimensional catenoid. This finishes the proof of the theorem.
\end{proof}

\bigskip

\section{The moving plane method on manifolds with symmetry}\label{sec_man}
The goal of this section is to prove Theorem \ref{alex_gen_thm}. In fact, our entire discussion so far readily generalizes to Riemannian manifolds which are symmetric with respect a moving plane.
\begin{theorem}\label{main_man_thm}
Consider a smooth Riemannian metric on $N=\mathbb{R}\times N_0$ such that
\begin{enumerate}
\item For every $c$, the metric is invariant under the reflection $F_c$, defined by $(x_1,p)\mapsto (2c - x_1,p)$, and
\item The vectorfield $\partial_{x_1}$ is orthogonal to the foliation $\{x_1=c\}$.
\end{enumerate}
Suppose $M$ is a tame CMC-varifold in $N$ without boundary
and with compact, connected support. Then $M$ is invariant under reflection across some $N_c:=\{x_1=c\}.$ Furthermore, the portion of $M$ on either side of that $N_c$ is a smooth graph over an open subset of $N_c$.
\end{theorem}
\begin{proof}
The proof of the Hopf Lemma for varifolds (Theorem \ref{mirror-theorem_intro}) and the strong maximum principle for varifolds (Theorem \ref{strong_max_varifold_intro}) are still valid in the context of any Riemannian manifold. The symmetry assumptions $(i),(ii)$ allows one to argue precisely as in the proof of Theorem \ref{main_boundary_combined}, to show that $\mathcal{S}=(c,\infty)$ for some $c$ such that $M$ is invariant under reflection across $N_c$.    
\end{proof}

Note that if $||\partial_{x_1}||=1$ in  Theorem \ref{main_man_thm}, then assumptions $(i),(ii)$ imply that  $N$  carries a product metric. In general, the metric on $N$ need not split an $\mathbb{R}$-factor, as we shall now exploit:

\begin{lemma}\label{lemma_hyp}
There exists a model to the hyperbolic space of the form $N=\mathbb{R}\times N_0$, satisfying (i) and (ii). 
\end{lemma}
\begin{proof}
Consider the upper half space model on $\mathbb{H}^{n+1}$, with the metric $g=\frac{dx_1^2+\ldots+dx_{n+1}^2}{x_1^2}$, and let $N=\mathbb{R}\times \mathbb{S}^n_+$, where $\mathbb{S}^n_+$ is the upper hemisphere of $n$-th sphere. Considering the diffeomorpishm $f:N\rightarrow \mathbb{H}^{n+1}$ defined by 
\begin{equation}
f(t,\omega)=e^t\omega,
\end{equation} 
one easily sees that the metric $f^{\ast}g$ is $t$-independent, and is such that $\partial_t$ is orthogonal to the hypersurfaces $\{t=c\}$.
\end{proof}

\begin{proof}[Proof of Theorem \ref{alex_gen_thm}]
Let us first consider the case $M\subset \mathbb{H}^{n+1}$. Since $M$ is compact, there exists a unique ball of minimal radius $B(p,R)\subset \mathbb{H}^{n+1}$ containing $M$. For each $n$-dimensional hyperplane in $P\subseteq T_p\mathbb{H}^{n+1}$, we have that $\tilde{P}:=\{\mathrm{exp}_p(v),\;v\in P\}$ is a totally geodesic hypersurface in $\mathbb{H}^{n+1}$. Let $\Psi$ be an isometry of $\mathbb{H}^{n+1}$ sending $p$ to some point $q\in N_0$, such that $d\Psi_p(P)=T_{q}N_0$. Since $N_0$ is also totally geodesic, we see that $\Psi$ sends $\tilde{P}$ to $N_0$. Now, Theorem \ref{main_man_thm}, which can be applied thanks to Lemma \ref{lemma_hyp}, implies that there exists some $c\in \mathbb{R}$ such that the $F_c(\Psi(M))=\Psi(M)$, so $M=\Psi^{-1}(F_c(\Psi(M)))$. As the isometry $\Phi:=\Psi^{-1}\circ F_c \circ \Psi$ fixes $M$, it also has to fix the ball $B(p,R)$, and so it fixes $p$. Thus $c=0$.
Note further that $d\Phi_p$ acts on $T_p\mathbb{H}^{n+1}$ by a reflection about $\tilde{P}$.
As reflections generate $\mathrm{SO}(n+1)$, and as isometries of $\mathbb{H}^{n+1}$ fixing $p$ are uniquely determined by their differential at $p$, we see that $M$ is smooth, and is fixed by all of the isometries of $\mathbb{H}^{n+1}$ that fix $p$.  Since  $M$ is connected, we conclude that $M$ is a a geodesic sphere. \\

For $M\subseteq \mathbb{H}^n\times \mathbb{R}$, the same argument shows that there exists some point $(p,x)\in \mathbb{H}^n\times \mathbb{R}$ such that $M$ is invariant under all isometries of the form $(\Psi, Id)$, where $\Psi$ is an isometry of $\mathbb{H}^n$ fixing $p$. Thus, $M$ is rotationally symmetric around $\{p\}\times \mathbb{R}$ and is smooth away from the axis. Arguing as in the proof of Corollary \ref{full_sym_cor} we also get smoothness along the axis of symmetry. 
\end{proof}

\bigskip

\bibliographystyle{alpha}
\bibliography{moving_planes}

\vspace{10mm}

{\sc Robert Haslhofer, Department of Mathematics, University of Toronto,  40 St George Street, Toronto, ON M5S 2E4, Canada}\\

{\sc Or Hershkovits, Institute of Mathematics, Hebrew University, Givat Ram, Jerusalem, 91904, Israel}\\

{\sc Brian White, Department of Mathematics, Stanford University, 450 Serra Mall, Stanford, CA 94305, USA}\\

\end{document}